\documentclass[10pt]{amsart}
\usepackage{amsmath,amssymb,amsthm}
\usepackage{graphicx}
\usepackage{color}
\newtheorem{theorem}{Theorem}    

\newtheorem{proposition}{Proposition}

\newtheorem{corollary}{Corollary}
\theoremstyle{definition}
\newtheorem{definition}[theorem]{Definition}

\newtheorem*{remark*}{Remark}

\newcommand{\R}{\mathbb{R}}

\newcommand{\mF}{\mathcal{F}}

 \makeatletter
    
    \@addtoreset{equation}{section}
  \makeatother

\title[A quantitative Birman-Menasco finiteness theorem]{A quantitative Birman-Menasco finiteness theorem and its application to crossing number}
\author[T.Ito]{Tetsuya Ito}
\address{Department of Mathematics, Kyoto University, Kyoto 606-8502, JAPAN}
\email{tetitoh@math.kyoto-u.ac.jp}
 
\begin{document}

\begin{abstract}
Birman-Menasco proved that there are finitely many knots having a given genus and braid index. We give a quantitative version of Birman-Menasco finiteness theorem, an estimate of the crossing number of knots in terms of genus and braid index. This has various applications of crossing numbers, such as, the crossing number of connected sum or satellites.
\end{abstract}

\maketitle

\section{Introduction}

In \cite{BM6} Birman-Menasco proved (in)finiteness theorem on closed braid representatives of knots and links; a knot or link $K$ has only finitely many closed braid representatives, up to conjugacy and exchange move. Here an \emph{exchange move} is an operation that converts an $n$-braid of the form $\alpha \sigma_{n-1} \beta \sigma_{n-1}^{-1}$ ($\alpha,\beta \in B_{n-2}$) to a braid $\alpha \sigma_{n-1}^{-1} \beta \sigma_{n-1}$ and its converse (see Figure \ref{fig:exchangemove}).
As a byproduct they proved another remarkable finiteness result;
\begin{theorem}[Birman-Menasco finiteness theorem \cite{BM6}]
\label{theorem:BM}
For given $g,n>0$, there are only finitely many knots/links with genus $g$ and braid index $n$.
\end{theorem}

\begin{figure}[htbp]
\begin{center}
\includegraphics*[bb=190 607 421 715,width=80mm]{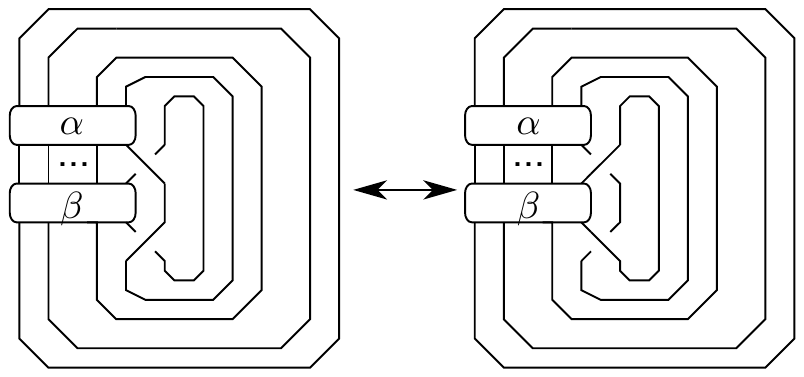}
\caption{Exchange move.}
\label{fig:exchangemove}
\end{center}
\end{figure}

In this paper, we give a quantitative version of Theorem \ref{theorem:BM}, an estimate of the crossing number by braid index and genus.
For a link $L$ in $S^{3}$, let $b(L),c(L)$ and $\chi(L)$ be the braid index, the minimum crossing number, and the maximum euler characteristic of Seifert surface, respectively.

\begin{theorem}[Quantitative Birman-Menasco finiteness theorem]
\label{theorem:main}
Let 
\[ f(n)= \begin{cases} 
1 & (n=2) \\
\frac{5}{3} & (n=3)\\
(2n-5)& (n>3).
\end{cases}\]
For a link $L$ in $S^{3}$,
\[ -\chi(L)+b(L) \leq c(L) \leq f(b(L))(-\chi(L)+b(L)) \]
holds.
\end{theorem}
The case $b(L)=2$ is easy and the case $b(L)=3$ is proven in \cite{It-fertile} by using a special property of closed 3-braid that a closed 3-braid always bounds a minimum genus Bennequin surface, a minimum genus Seifert surface whose braid foliation has only aa-singular points \cite{Be,BM2}.

In the following, we restrict our attention to knots, and use the knot genus instead of the maximum euler characteristic. 

Theorem \ref{theorem:main} implies that $2g(K)-1+b(K)$ is a linear approximation of the crossing number. This has several applications to the crossing numbers since the genus is one of the most familiar and well-studied invariant whereas the minimum crossing number is one of the most difficult invariant whose basic properties are still unknown.

It is a famous conjecture that $c(K\#K')=c(K)+c(K')$ \cite[Problem 1.65]{Ki}. 
Theorem \ref{theorem:main} immediately gives an estimate of the crossing number of composite knots.

\begin{corollary}
\label{cor:composite}
For a knot $K,K'$ in $S^{3}$ let $N=\max\{b(K),b(K')\}$.
\[ c(K\#K') \geq \frac{1}{f(b(K))}c(K) + \frac{1}{f(b(K'))}c(K') \geq \frac{1}{f(N)}(c(K)+c(K'))\]
\end{corollary}
\begin{proof}
Recall that $g(K)$ is additive under connected sum, and that $b(K)-1$ is also additive under connected sum \cite{BM4}.
Thus 
\begin{align*}
c(K\# K') & \geq b(K\#K') +2g(K\# K') -1\\
& = (2g(K)-1+b(K))+(2g(K')-1+b(K'))\\
&\geq \frac{c(K)}{f(b(K))}+\frac{c(K')}{f(b(K'))}\\
& \geq \frac{1}{f(N)}(c(K)+c(K'))
\end{align*}
\end{proof}
In \cite{La1} Lackenby showed that $c(K\#K') \geq \frac{1}{152}(c(K)+c(K'))$. Although the estimate given in Corollary \ref{cor:composite} gets worse as $b(K)$ increases, when $b(K)$ is small it gives a better estimate.

Our proof gives an estimate of regularity introduced in \cite{Ma} which plays an important role to relate two different problem; a genericity of hyperbolic knots/links and the additivity of crossing numbers. 
\begin{definition}
For $\lambda \in \R$, a knot $K$ is \emph{$\lambda$-regular} if $c(K') \geq \lambda c(K)$ if $K'$ contains $K$ as its prime factor.
\end{definition}

\begin{corollary}
\label{cor:regularity}
For a knot $K$ in $S^{3}$, $K$ is $\frac{1}{f(b(K))}$-regular.
\end{corollary}

Similarly, we have an estimate of the crossing number of satellite knots and an estimate of the asymptotic crossing number.
The \emph{asymptotic crossing number} $ac(K)$ of $K$ is defined by\footnote{The definition of the asymptotic crossing number given here is a version appeared in Kirby's problem list \cite{Ki}. In \cite{FH} Freedman-He called this the \emph{quadratic asymptotic crossing number} and use different definition for the assymptotic crossing number, which we denote by $ac_{FH}(K)$. Since Freedman-He proved $ac(K) \geq ac_{FH}(K) \geq 2g(K)-1$, the inequality in Corollary \ref{cor:satellite} is valid for $ac_{FH}(K)$.}
\[ ac(K)= \inf_{w\geq 1} \left\{\frac{c(K_w)}{w^{2}}\right\} \]
where the infimum is taken over all satellites $K_w$ of $K$ with winding number $w \geq 0$. It is conjectured $ac(K)=c(K)$ \cite[Problem 1.68]{Ki}.

\begin{corollary}
\label{cor:satellite}
Let $K$ be a satellite knot with companion knot $K_0$ and pattern $C$, whose winding number (homological degree) is $w$. Then
\[ c(K) \geq \frac{w^2}{f(b(K_0))}c(K_0)  - w^2 b(K_0)\]
In particular, 
\[ ac(K) \geq \frac{1}{f(b(K))}c(K) -b(K).\]
\end{corollary}
\begin{proof}
By \cite{FH} it is shown that $c(K) \geq w^{2}(2g(K_0)-1)$. Therefore
\begin{align*}
c(K) & \geq  w^{2}(2g(K_0)-1)\geq w^{2}\left( \frac{1}{f(b(K_0))}c(K_0)  - b(K_0)\right)
\end{align*}
\end{proof}
 Although Corollary \ref{cor:satellite} is meaningful only if $b(K)$ is small compared with $c(K)$, it gives a linear lower bound of $ac(K)$ in terms of $c(K)$.

A famous \emph{satellite crossing number conjecture} asserts $c(K) \geq c(K_0)$ \cite[Problem 1.67]{Ki}. Corollary \ref{cor:satellite} says that the satellite crossing number conjecture is true when $w$ and $c(K_0)$ are large enough compared with $b(K_0)$. 

A weaker and more tractable version of the satellite crossing number conjecture, $c(K_{p,q}) \geq c(K)$ for the $(p,q)$-cable $K_{p,q}$ of $K$ which we call \emph{the cabling crossing number conjecture}, is still open. For a braided $p$-cable, a satellite knot whose pattern $C$ is a closed $p$-braid in the solid torus, we have another estimate that gives more supporting evidences for the cabling crossing number conjecture and the satellite crossing number conjecture which works even if $b(K)$ is large;
\begin{corollary}
\label{cor:braided-cable}
If $K_p$ is a braided $p$-cable of $K$, then 
\[ c(K_p) \geq \frac{p}{f(b(K))}c(K)+(p-1)\]
Thus, if $p \geq f(b(K))$ then $c(K_p)\geq c(K)$. 
\end{corollary}
\begin{proof}
Since $g(K_p)=pg(K)+g(C) \geq pg(K)$ and $b(K_p)=pb(K)$ \cite{Wi} we get
\begin{align*}
c(K_p) & \geq 2g(K_p)-1+b(K_p) \geq p(2g(K)-1+b(K)) + (p-1)\\
& \geq \frac{p}{f(b(K))}c(K) + p-1
\end{align*}
\end{proof}

In particular, since $f(3)=\frac{5}{3}<2$, the cabling crossing number conjecture is true for closed 3-braids:
\begin{corollary}
If $b(K) \leq 3$ then $c(K_{p,q}) \geq c(K)$ for all $(p,q)$-cables $K_{p,q}$ of $K$.
\end{corollary}

It is interesting to ask whether Corollary \ref{cor:braided-cable} can be extended to general satellite knot. We used the braided $p$-cable assumption is to guarantee $b(K_p)=pb(K)$. It looks to be reasonable to expect $b(K_w) \geq wb(K)$ for satellite $K_w$ of $K$ with winding number $w$ so the same conclusion as Corollary \ref{cor:braided-cable} holds for general satellite knot $K_w$ with winding number $w$. 

In \cite{La2} Lackenby showed $c(K) \geq \frac{1}{10^{13}}c(K_0)$ if $K$ is a satellite of $K_0$. Since this $\frac{1}{10^{13}}$ estimate uses an estimate for $(2,t)$-cable of $K_{2,t}$ of $K_0$ and Corollary 
\ref{cor:braided-cable} gives an improvement of Lackenby's estimate $\frac{1}{119024}(c(K_0)-|t|)\leq c(K_{2,t})$ \cite[Theorem 5.1]{La2}, one can improve Lackenby's $\frac{1}{10^{13}}$ bound when $b(K_0)$ is not so large. Indeed, combining Lackenby's another result and our result we give a dramatic improvement of the estimate of the crossing number of satellite knots when $b(K_0)$ is not large.

\begin{corollary}
\label{cor:satellite2}
If $K$ is a satellite of $K_0$, 
\[ c(K) \geq \frac{1}{76f(b(K_0))}c(K_0)\]
\end{corollary}
\begin{proof}
By \cite[Theorem 1.5]{La2}, there is a knot $L$ which is either $K_0$ or a cable of $K_0$ such that $c(K) \geq \frac{1}{152}c(L)$. 
By Corollary \ref{cor:braided-cable}, if $L$ is a cable of $K_0$ then $c(L) \geq \frac{2}{f(b(K_0))}c(K_0)$ hence
\[
c(K) \geq \frac{1}{152} c(L) \geq \frac{1}{152}\left(\frac{2}{f(b(K_0))}c(K_0)\right) = \frac{1}{76f(b(K_0))}c(K_0)\]
\end{proof}

Finally we give another application. The \emph{braid index problem} is a problem to algorithmically compute the braid index of a link $L$ from a given diagram $D$ of $L$. Strictly speaking, to make it as a decision problem, by the braid index problem we mean the problem to determine, for a given diagram $D$ of $L$ and a natural number $n$, determine whether $b(L) \leq n$ or not. Since $K$ is the unknot if and only if $b(K)=1$, this problem includes the \emph{unknotting problem}, a problem to recognize the unknot.

\begin{theorem}
The braid index problem is solvable.
\end{theorem}
\begin{proof}
Our proof of Theorem \ref{theorem:main} actually shows that a link $L$ has a closed $b(L)$-braid diagram that has at most $f(b(L))(-\chi(L)+b(L))$ crossings. There are finitely many closed $n$-braid diagram with less than or equal to $c$ crossings. Since one can effectively enumerate such closed braid diagrams, by checking all the possibilities (this is possible since one can algorithmically check a given diagram $D$ represents $L$ or not \cite{Ha1,Ha2,He}) one can determine $b(L) \leq n$ or not algorithmically.
\end{proof}

It is desirable to give more direct algorithm of the braid index problem that avoids to use knot recognition, in a spirit of the Birman-Hirsch unknotting algorithm \cite{BH}. Also, it is an interesting problem to determine the computational complexity of the braid index problem.

\section*{Acknowledgement}
The author has been partially supported by JSPS KAKENHI Grant Number 19K03490,16H02145. He would like to thank Joan Birman for stimulating comments for earlier version of the paper, and would like to thank to Sorsen Rebecca and Keiko Kawamuro pointing out an error in the earlier version of the paper.

\section{Proof of Theorem \ref{theorem:main}}

\subsection{Quick review of braid foliation}
In this section we briefly review Birman-Menasco's braid foliation techniques. For details of braid foliation theory, we refer to \cite{LM} or \cite{BF}.

Let $A$ be the unknot in $S^{3}$ and fix a fibration $\pi:S^{3} \setminus A \rightarrow S^{1}$ whose fiber $S_{t}=\pi^{-1}(t)$ is a disk bounded by $A$. An oriented link $L$ is a \emph{closed $n$-braid} with axis $A$ if $L$ positively transverse to every fiber $S_{t}$ at $n$ points.

Let $F$ be an oriented incompressible Seifert surface of a closed $n$-braid $L$. We put $F$ so that $\{F\cap S_{t}\}$ induces a singular foliation $\mathcal{F}$ of $F$ satisfying the following properties;
\begin{itemize}
\item[(i)] $A$ and $F$ transversely intersect. Moreover, each $v \in A \cap F$ is  an \emph{elliptic} singular point of $\mathcal{F}$; there is a disk neighborhood $U_v$ of $v$ in $F$ such that the foliation $\mF$ on $U_v$ is the  radial foliation with node $v$.
\item[(ii)] All but finitely many fibers $S_{t}$ transversely intersect with $F$. Each exceptional fiber $S_{t}$ is tangent to $F$ at exactly one point $h$ in the interior of $F$, as a saddle tangency that appears as a \emph{hyperbolic} singular point of $\mathcal{F}$.
\item[(iii)] Each leaf of $\mathcal{F}$, a connected point of $S_{t}\cap F$, transverse to $L=\partial F$.
\end{itemize}
We call the foliation $\mathcal{F}$ satisfying these properties \emph{braid foliation}.

An elliptic singular point $v$ is \emph{positive} (resp. negative) if the sign of intersection of $A \cap F$ at $v$ is positive (resp. negative). A hyperbolic singular point $h \in F\cap S_{t}$ is \emph{positive} (resp. negative) if the positive normal direction $\vec{n}_F$ of $F$ at $h$ agrees (resp. desagrees) with that of $S_{t}$ at $h$ (see Figure \ref{fig:sign})\footnote{Although in the following argument the sign of hyperbolic point does not appear explicitly, it plays a crucial role in the proof of Proposition \ref{prop:BM} below.}.

\begin{figure}[htbp]
\begin{center}
\includegraphics*[bb= 163 551 442 715,width=90mm]{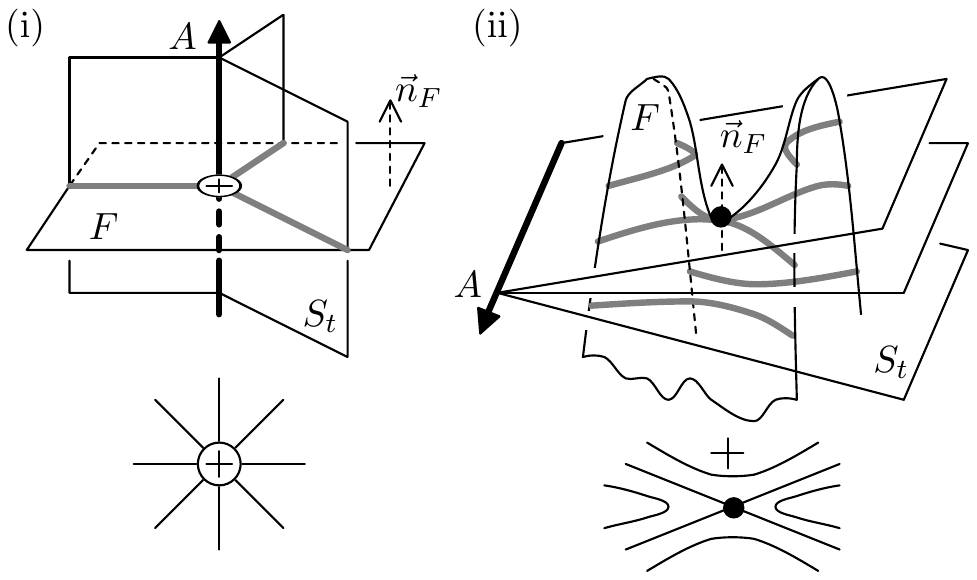}
\caption{Singular points and its signs of braid foliation}
\label{fig:sign}
\end{center}
\end{figure}

If a connected component $c$ of $S_{t} \cap F$ is a simple closed curve, it bounds a disk in $S_{t}$. Since $F$ is incompressible, by standard innermost circle argument one can remove such leaves by isotopy of $F$. 
Thus in the following, we always assume that the leaf of $\mathcal{F}$ is either
\begin{itemize}
\item \emph{a-arc}: an arc connecting a positive elliptic point and a point of $L=\partial F$, or,
\item \emph{b-arc}: an arc connecting two elliptic points of opposite signs.
\end{itemize} 
Moreover, with no loss of generality, we may assume that each b-arc $b$ in $S_t$ is \emph{essential} as a properly embedded arc in the punctured disk $S_t \setminus (S_t \cap L)$. In other words, both components of $S_t \setminus b$ are pierced by the braid $L=\partial F$.

According to the types of nearby leaves, hyperbolic points of $\mF$ are classified into three types, $aa$, $ab$ and $bb$ and each hyperbolic point has a canonical foliated neighborhood as shown in Figure \ref{fig:tiles} which we call \emph{$aa$-tile, $ab$-tile, and $bb$-tile}. 

\begin{figure}[htbp]
\begin{center}
\includegraphics*[bb=190 627 421 715,width=75mm]{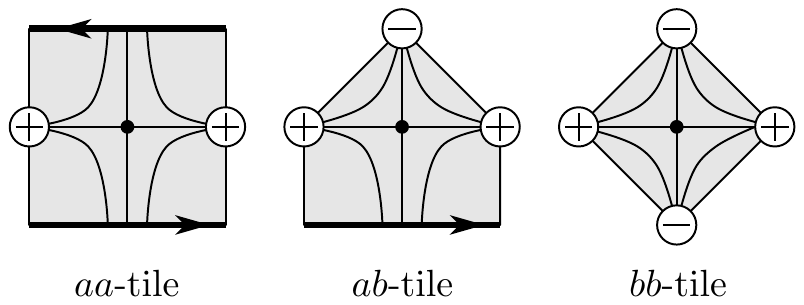}
\caption{Tiles}
\label{fig:tiles}
\end{center}
\end{figure}

A decomposition of $F$ into tiles induces a cellular decomposition of $F$;
The $2$-cells are tiles, the $1$-cells are $a$-arcs and $b$-arcs that appear as the boundary of tiles, and the $0$-cells are elliptic points and the end points of $a$-arc 1-cells.

\subsection{Euler characteristic equality}

We say that an elliptic point is of \emph{type $(\alpha,\beta)$} if the valence  is $\alpha+\beta$ and it is a boundary of $\alpha$ $a$-arc 1-cells and $\beta$ $b$-arc $1$-cells. 
We denote by $V(\alpha,\beta)$ be the number of elliptic points of type $(\alpha,\beta)$.

By a standard euler characteristic counts we have the following equality which we call the \emph{euler characteristic equality}, or, the \emph{fundamental equality of braid foliation}; 
\begin{align}
\label{eqn:euler}
&2V(1,0)+2V(0,2)+V(0,3)+V(1,1)-4\chi(F)  \\
& \qquad \qquad = V(2,1)+2V(3,0)+\sum_{v\geq4}\sum_{\alpha=0}^{v}(v+\alpha-4)V(\alpha,v-\alpha) \nonumber
\end{align}

The following result plays the fundamental role in the braid foliation theory and explain why exchange move is a fundamental operation in braid foliation; using exchange move one can simplify the braid foliation (or braid itself, by destabilization) so that it contains no low-valency vertices.
In a light of the euler characteristic equality (\ref{eqn:euler}), this gives a strong constraint for $V(\alpha,\beta)$.

\begin{proposition}\cite[Lemma 4]{BM6}
\label{prop:BM}
If $n=b(L)$, by applying exchange move we may assume that
$V(1,0)=V(0,2)=V(0,3)=V(1,1)=0$.
\end{proposition}


\subsection{Braid foliation and crossing number}

One can read the braid from the braid foliation \cite{BH,BM1,It-genus}; one $aa$-tile gives rise to a braid of the form
\[\sigma_{i,j}:=(\sigma_{i}\sigma_{i+1}\cdots \sigma_{j-2})\sigma_{j-1}^{\pm 1}(\sigma_{i}\sigma_{i+1}\cdots \sigma_{j-2})^{-1} \quad (1\leq i<j \leq n)\]
and one $ab$-tile gives rise to a braid of the form,
\[ (\sigma_{i}\sigma_{i+1} \cdots \sigma_{j-1})^{\pm 1}, \mbox{or, } (\sigma_{j-1}\sigma_{j-2}\cdots \sigma_{i})^{\pm 1} \quad (1\leq i<j \leq n).\]
A bb-tile does not affect the boundary the braid $L=\partial F$.

One useful way to see this is to use a movie presentation, a sequence of slice $S_t \cap F$ that describes how the arcs in $S_t$ changes as $t$ increase as we illustrate in Figure \ref{fig:braidmovie}.

\begin{figure}[htbp]
\begin{center}
\includegraphics*[bb=169 486 441 715,width=100mm]{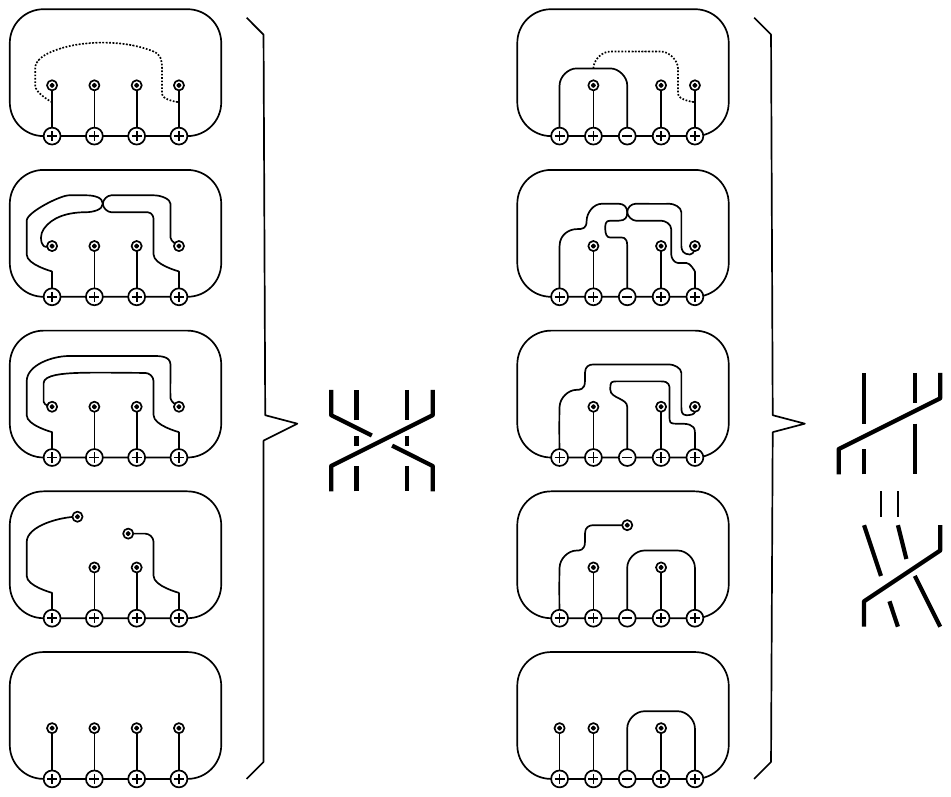}
\caption{How aa- and ab-singularity affect the braid $L=\partial F$}
\label{fig:braidmovie}
\end{center}
\end{figure}

Let $R_{aa}$ and $R_{ab}$ be the number of $aa$ and $ab$ tiles in the braid foliation, respectively. Since the above discussion shows that one aa-tile provides a braid with at most $2(n-2)+1=2n-3$ crossings and one ab-tile provides a braid with at most $(n-1)$ crossings, we have an estimate of the crossing number via the braid foliation
\[ c(L) \leq(2n-3) R_{aa}+ (n-1)R_{ab}.\]
With a little more argument we slightly improve the estimate as follows;

\begin{proposition}
\label{prop:braidfol-cross}
Let $L$ be a closed $n$-braid  and $F$ be an incompressible Seifert surface admitting a braid foliation. 
\begin{enumerate}
\item[(i)] If $R_{ab} =0$ then 
\[ c(L) \leq \begin{cases}
R_{aa} & (n=2)\\
\frac{5}{3}R_{aa} & (n=3) \\
(2n-5)R_{aa} & (n>4) \\
\end{cases}\]
\item[(ii)] If $R_{ab} \neq 0$ then
\[ c(L) \leq (2n-5)R_{aa} + (n-3) R_{ab} \]
\end{enumerate}
\end{proposition}
\begin{proof}
(i): Assume that $R_{ab}=0$. Then case $n=2$ is obvious so we assume that $n>2$. Let $R_1,R_2,\ldots,R_{n}$ be the number of aa-tiles that gives rise to a braid $\sigma_{i,i+1}$ and let $R'=R_1+R_2+\cdots +R_n$.
Here indices are understood by modulo $n$ so we regard $\sigma_{n,n+1}$ as $\sigma_{1,n}$.  
Since 
\[ (\sigma_1\sigma_2\cdots \sigma_{n-1})\sigma_{i,j}(\sigma_1\sigma_2\cdots \sigma_{n-1})^{-1} =\sigma_{i+1,j+1},\]
by taking conjugates we may assume that $R_{n} \leq R_1,\ldots,R_{n-1}$.
Thus $R_n \leq \frac{1}{n} R'$. 
Each ab-tile counted in $R_1,\ldots,R_{n-1}$ provides a braid with one crossing.
Each of rest of the aa-tiles provides a braid with at most $2(n-1)-3=2n-5$ crossings so 
\begin{align*}
c(L) &\leq (2n-3)\frac{R'}{n}+ \frac{n-1}{n}R' + (2n-5) (R_{aa}-R')\\
& = \frac{3n-4}{n} R' +(2n-5)(R_{aa}-R')\\
& \leq \begin{cases} 
\frac{3n-4}{n}R_{aa} & (n=3) \\
(2n-5)R_{aa} & (n\geq 4) \\
\end{cases}
\end{align*}

(ii) When $R_{ab}\neq 0$, the braid foliation of $F$ contains at least one negative elliptic point. We put a negative elliptic point $v$ so that it is leftmost among all the ellitpic poionts. At each $t$, a b-arc from $v$ always separates $S_t$ into two components so that both component contains at least one puncture point $L \cap S_t$. Since $v$ is leftmost, the leftmost puncture point, the 1st strand of the braid, and the rightmost puncture point, the $n$-th strand of the braid, are always separated by the b-arc from $v$. Hence no aa-tile can provide a braid $\sigma_{1,n}$, so one $aa$-tile can provide a braid having at most $2n-5$ crossings (see Figure \ref{fig:restriction} (i)).

By the same reason no $ab$-tile can provide a braid of the form 
\[ (\sigma_1\cdots \sigma_{n-1})^{\pm 1}, (\sigma_{n-1}\cdots \sigma_{1})^{\pm 1}.\]
Moreover, an $ab$-tile cannot produce a braid of the form
\[ (\sigma_1\cdots \sigma_{n-2})^{\pm 1}, (\sigma_2\cdots \sigma_{n-1})^{\pm 1}, (\sigma_{n-2}\cdots \sigma_{1})^{\pm 1}, (\sigma_{n-1}\cdots \sigma_{2})^{\pm 1}\]
either, because if such a braid is produced by an $ab$-tile, the b-arc from $v$ becomes boundary-parallel (see Figure \ref{fig:restriction} (ii)). Thus one ab-tile actually can provide a braid having at most $(n-3)$ crossings.
Therefore we conclude
\[ c(L) \leq (2n-5)R_{aa}+ (n-3) R_{ab}.\]

\begin{figure}[htbp]
\begin{center}
\includegraphics*[bb=148 555 448 723,width=90mm]{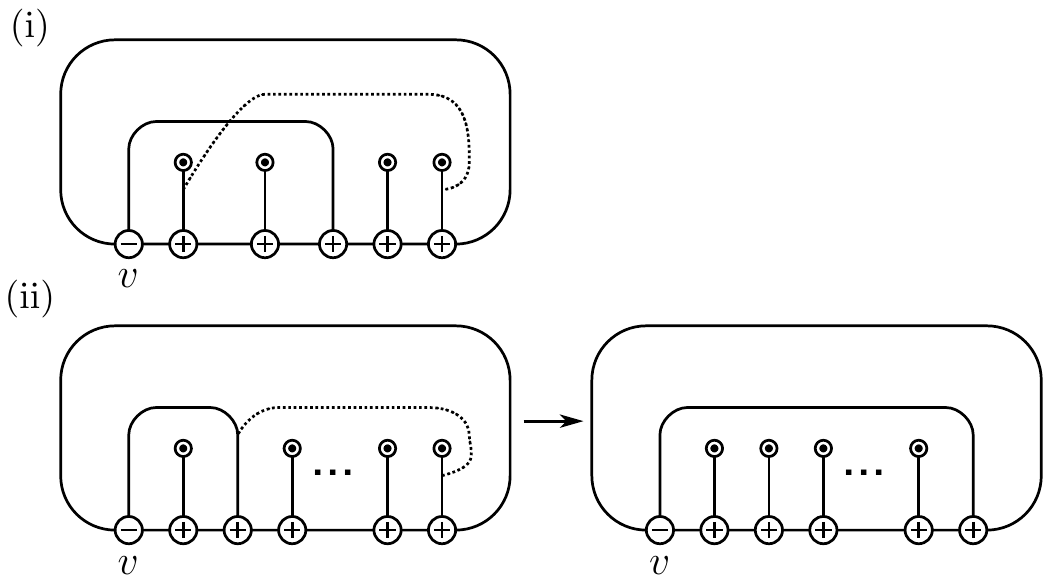}
\caption{(i): The b-arc from the leftmost negative elliptic point $v$ prevents to form an $aa$-singular point producing $\sigma_{1,n}$. (ii): An $ab$-singular point producing $(\sigma_{n-1}^{-1}\cdots \sigma_{2}^{-1})$ makes the b-arc from $v$ boundary-parallel.}
\label{fig:restriction}
\end{center}
\end{figure}

\end{proof}

\subsection{Completion of proof}

Once all the necessary backgrounds or results from braid foliation theory are prepared, the proof of Theorem \ref{theorem:main} is essentially an elementary counting argument.

\begin{proof}[Proof of Theorem \ref{theorem:main}]

Let $D$ be a minimum crossing diagram of $L$. By Seifert's algorithm, we get a Seifert surface $\Sigma_D$ of $L$ with $\chi(\Sigma_D)=s(D)-c(L)$, where $s(D)$ denotes the number of Seifert circles. Since $b(L) \leq s(D)$ we get
\[ \chi(L)\geq \chi(\Sigma_D) = s(D)-c(L) \geq b(L)-c(L).\] 

We prove the upper bound of $c(L)$. By Proposition \ref{prop:braidfol-cross}, it is sufficient to give an upper bound of $R_{aa}$ and $R_{ab}$. In the following argument, we give an estimate of $2R_{aa}+R_{ab}$.

Let $F$ be a Seifert surface of $L$ admitting a braid foliation such that $\chi(F)=\chi(L)$ and that $\partial F$ is a closed $b(L)$-braid.
By Proposition \ref{prop:BM} we may assume that $V(1,0)=V(0,2)=V(0,3)=V(1,1)=0$. In particular, $V(\alpha,v-\alpha)=0$ whenever $\alpha+v-4<0$.

Let $V^{\pm}(\alpha,\beta)$ be the number of positive and negative elliptic points of type $(\alpha,\beta)$. Note that $V^{+}(\alpha,\beta)=V(\alpha,\beta)$ whenever $\alpha>0$ since the elliptic point is a boundary of $a$-arc only if it is positive. Since $\partial F$ is a closed $b(L)$-braid, the algebraic crossing number of the axis $A$ and $F$ is $b(L)$ so
\begin{equation}
\label{eqn:braid-index}
b(L) = \sum_{v\geq 0}\sum_{\alpha=0}^{v} V^{+}(\alpha,v-\alpha)-V^{-}(\alpha,v-\alpha)
\end{equation}

Each tile contains exactly two positive ellitpic points so
\[ 2R_{aa}+2R_{ab}+2R_{bb} = \sum_{v\geq 0}\sum_{\alpha=0}^{v} vV^{+}(\alpha,v-\alpha).\]
Similarly, each ab-tile contains one negative elliptic point, and each bb-tile contains two negative elliptic points so 
\[ R_{ab}+2R_{bb} = \sum_{v\geq 0}\sum_{\alpha=0}^{v} vV^{-}(\alpha,v-\alpha).\]
Thus 
\begin{equation}
\label{eqn:tile-vertex}
2R_{aa}+R_{ab} =  \sum_{v\geq 0}\sum_{\alpha=0}^{v}vV^{+}(\alpha,v-\alpha) -v V^{-}(\alpha,v-\alpha).
\end{equation} 

By (\ref{eqn:tile-vertex}) and (\ref{eqn:braid-index}) we get
\begin{align*}
2R_{aa}+R_{ab}-4b(L) &= \sum_{v\geq 0}\sum_{\alpha=0}^{v}(v-4)V^{+}(\alpha,v-\alpha) -(v-4)V^{-}(\alpha,v-\alpha)\\
&\hspace{-1.5cm}= \sum_{v\geq 0}\sum_{\alpha=1}^{v}(v-4)V(\alpha,v-\alpha) + \sum_{v> 0}(v-4)V^{+}(0,v)- \sum_{v> 0}(v-4)V^{-}(0,v).
\end{align*}
Since we have seen that $V(0,v)=0$ for $v <4$, 
\begin{equation}
\label{eqn:tile1}
2R_{aa}+R_{ab}-4b(L) \leq  \sum_{v\geq 0}\sum_{\alpha=0}^{v}(v-4)V(\alpha,v-\alpha)
\end{equation}

On the other hand, let $E_a$ be the number of 1-cells which are $a$-arcs. By counting $E_a$ as a boundary of 2-cells, $E_a = 2R_{aa}+R_{ab}$. Similarly by counting $E_a$ as edges connecting 0-cells we get
\begin{equation}
\label{eqn:edge}
2R_{aa}+R_{ab} = E_a = \sum_{v>0}\sum_{\alpha=0}^{v}\alpha V(\alpha,v-\alpha)
\end{equation}
Thus by (\ref{eqn:tile1}), (\ref{eqn:edge}), and the euler characteristic equality (\ref{eqn:euler}) 
\[ 2(2R_{aa}+R_{ab})-4b(L) \leq \sum_{v>0}\sum_{\alpha=0}^{v}(v+\alpha-4) V(\alpha,v-\alpha) = -4\chi(L) \]
so we eventually arrive at the inequality
\[ 2R_{aa}+R_{ab} \leq -2\chi(L)+2b(L) \]
By Proposition \ref{prop:braidfol-cross} we conclude
\[ c(L) \leq \begin{cases}
(-\chi(F)+b(L)) & b(L)=2 \\
\frac{5}{3}(-\chi(F)+b(L)) & b(L) = 3, \\
(2b(L)-5)(-\chi(F)+b(L)) & b(L) > 4.
\end{cases}
\]
\end{proof}

\end{document}